\documentclass{amsart}

\usepackage{setspace, amsmath, amsthm, amssymb, amsfonts, amscd, epic, graphicx, ulem, dsfont}
\usepackage[T1]{fontenc}
\usepackage{multirow}
\usepackage{bbm}

\newtheorem{thm}{Theorem}

\newtheorem{exa}{Example}

\theoremstyle{remark}

\theoremstyle{definition}

\newcommand{\R}{\mathbb{R}}

\begin{document}

\title{Products of Unbounded Normal Operators}
\author[M. H. MORTAD]{MOHAMMED HICHEM MORTAD}

\dedicatory{}
\date{}


\address{
Département de Mathématiques, Université d'Oran, B.P. 1524, El
Menouar, Oran 31000, Algeria.\newline {\bf Mailing address}:
\newline Dr Mohammed Hichem Mortad \newline BP 7085 Seddikia Oran
\newline 31013 \newline Algeria}

\email{mhmortad@gmail.com, mortad@univ-oran.dz.}

\begin{abstract}
The present paper partly constitutes an "unbounded" follow-up of a
paper by I. Kaplansky dealing with bounded products of normal
operators. Results on the normality of unbounded products are also
included.
\end{abstract}

\maketitle

\section{Introduction}
In this paper we are concerned with the normality of the products
$AB$ and $BA$ for two operators, where $B$ is unbounded. The
problems of this sort lie among the most fundamental questions in
the Hilbert space theory and were explored by many authors,
especially in the bounded case (see
\cite{Kapl,kitt-prod-norm,Wieg-matrices-normal,Wieg-infinite-matrices}).
See also the paper \cite{Gheondea} and the references therein.
Nonetheless, only a few shy attempts where made in the unbounded
case. See the very recent papers \cite{Gustafson-Mortad} and
\cite{MGM}.

We also cite the reference \cite{Mortad-Demmath} where the following
result (among others) has been obtained
\begin{thm}\label{mortad-demmath-normality}\hfill
\begin{enumerate}
  \item Assume that $B$ is a unitary operator. Let $A$ be an unbounded
normal operator. If $B$ and $A$ commute (i.e. $BA\subset AB$), then
$BA$ is normal.
  \item Assume that $A$ is a unitary operator. Let $B$ be an unbounded
normal operator. If $A$ and $B$ commute (i.e. $AB\subset BA$), then
$BA$ is normal.
\end{enumerate}
\end{thm}

For results involving normality and self-adjointness of unbounded
operator products, see
\cite{MHM1,Mortad-IEOT,Mortad-Demmath,Gustafson-Mortad}. For similar
papers on the sum of two normal operators, see
\cite{mortad-CAOT-sum-normal} and \cite{mortad-sum-newest}.

One purpose of this paper is to try to get an analog for unbounded
operators of the following result
\begin{thm}\label{Kapl-bounded}[Kaplansky, \cite{Kapl}]
Let $A$ and $B$ be two bounded operators on a Hilbert space such
that $AB$ and $A$ are normal. Then $B$ commutes with $AA^*$ iff $BA$
is normal.
\end{thm}

In order to do that and to allow a broader audience to read the
present paper, we recall basic definitions and results on unbounded
operators. Some important references are
\cite{CON,GGK,Goldb-1966-unbd,RUD}.

All operators are assumed to be densely defined (i.e. having a dense
domain) together with any operation involving them or their
adjoints. Bounded operators are assumed to be defined on the whole
Hilbert space. If $A$ and $B$ are two unbounded operators with
domains $D(A)$ and $D(B)$ respectively, then $B$ is called an
extension of $A$, and we write $A \subset B$, if $D(A) \subset D(B)$
and if $A$ and $B$ coincide on $D(A)$. We write $A\subseteq B$ if
$A\subset B$ or $A=B$ (meaning that $D(A)=D(B)$ and $Ax=Bx$ for all
$x\in D(A)$).

If $A\subset B$, then $B^* \subset A^*$. An unbounded operator $A$
is said to be closed if its graph is closed; self-adjoint if $A =
A^*$ (hence from known facts self-adjoint operators are
automatically closed); normal if it is closed and $AA^* = A^*A$
(this implies that $D(AA^* ) = D(A^*A)$).

A densely defined operator $A$ is said to be hyponormal if
$D(A)\subset D(A^*)$  and $\|A^*x\|\leq\|Ax\|$ for $x\in D(A)$. We
say that a densely defined operator $S$ in a Hilbert space $H$ is
subnormal if there is another Hilbert space $L\supset H$ and a
normal operator $N$ in $L$ such that $S\subset N$. It is well{known
that each subnormal operator is hyponormal and that each hyponormal
operator is closable

The Fuglede-Putnam theorem (see \cite{FUG} and \cite{PUT}) is
important to prove our results so we recall it here
\begin{thm}
Let $A$ be a bounded operator. Let $N$ and $M$ be two unbounded
normal operators. If $AN\subseteq MA$, then $AN^*\subseteq M^*A$.
\end{thm}

We digress a little bit to say that a new version of this famous
theorem has been obtained by the author where all the operators
involved are unbounded. See \cite{Mortad-Fuglede-Putnam-CAOT-2011}.

 It is also convenient to recall
the following theorem which appeared in \cite{STO}, but we state it
in the form we need.
\begin{thm}\label{STOCHEL-ASYMMETRIC}
If $T$ is a closed subnormal (resp. closed hyponormal) operator and
$S$ is a closed hyponormal (resp. closed subnormal) operator
verifying $XT^*\subset SX$ where $X$ is a bounded operator, then
both $S$ and $T^*$ are normal once $\ker X=\ker X^*=\{0\}$.
\end{thm}

\section{Main Results}

We start by giving a counterexample that shows that the same
assumptions, as in Theorem \ref{Kapl-bounded}, would not yield the
same results if $B$ is an unbounded operator, let alone the case
where both operators are unbounded.

What we want is a normal bounded operator $A$ and an unbounded (and
closed) operator $B$ such that $BA$ is normal, $A^*AB\subset BA^*A$
but $AB$ is not normal.
\begin{exa}\label{example counterexample main}
Let
\[Bf(x)=e^{x^2}f(x) \text{ and } Af(x)=e^{-x^2}f(x)\]
on their respective domains
\[D(B)=\{f\in L^2(\R):~e^{x^2}f\in L^2(\R)\} \text{ and } D(A)=L^2(\R).\]
Then $A$ is bounded and self-adjoint (hence normal). $B$ is
self-adjoint (hence closed).

Now $AB$ is not normal for it is not closed as $AB\subset I$. $BA$
is normal as $BA=I$ (on $L^2(\R)$). Hence $AB\subset BA$ which
implies that
\[AAB\subset ABA\Longrightarrow AAB\subset ABA\subset BAA.\]
\end{exa}

Now, we state and prove the generalization of Theorem
\ref{Kapl-bounded}. We have
\begin{thm}
Let $B$ be an unbounded closed operator and $A$ a bounded one such
that $AB$ (resp. $BA$) and $A$ are normal.  Then
\[BA \text{ normal (resp. $AB$)} \Longrightarrow A^*AB\subset BA^*A.\]
\end{thm}

\begin{proof}
Since $AB$ and $BA$ are normal, the
  equation
  \[A(BA)=(AB)A\]
  implies that
  \[A(BA)^*=(AB)^*A\]
  by the Fuglede-Putnam theorem. Hence
  \[AA^*B^*\subset B^*A^*A \text{ or } A^*AB\subset BA^*A.\]
\end{proof}

We already observed in Example \ref{example counterexample main}
that the converse in the previous theorem does not hold. An extra
hypothesis combined with a result by Stochel \cite{STO} yield the
following

\begin{thm}
If $B$ is an unbounded closed operator and and if $A$ is a bounded
one such that $AB$ and $A$ are normal, and if further $BA$ is
hyponormal (resp. subnormal), then

\[BA \text{ normal} \Longleftarrow A^*AB\subset BA^*A.\]
\end{thm}

\begin{proof}The idea of proof is similar in core to Kaplansky's (\cite{Kapl}). Let $A=UR$ be the polar decomposition of
  $A$, where $U$ is unitary and $R$ is positive (remember that they also
  commute and that $R=\sqrt{A^*A}$), then one may write
  \[U^*ABU=U^*URBU=RBU\subset BR U=BA\]
  or
  \[U^*AB=U^*\overline{AB}=U^*((AB)^*)^*\subset BA U^*\]
  (by the closedness of $AB$). Since $(AB)^*$ is normal, it is
  closed and subnormal. Since $B$ is closed and $A$ is bounded, $BA$
  is closed. Since it is hyponormal, Theorem \ref{STOCHEL-ASYMMETRIC} applies and
  yields the normality of $BA$ as $U$ is invertible.

  The proof is very much alike in the case of subnormality.
\end{proof}

Now, if we assume that $A$ is unitary, then we have the following
interesting result (cf Theorem \ref{mortad-demmath-normality}) that
bypasses the commutativity of operators. Besides, this constitutes
generalization of Theorem \ref{Kapl-bounded} with the assumption $A$
unitary.

\begin{thm}
If $A$ is unitary and $B$ is an unbounded normal operator, then
\[BA \text{ is normal $\Longleftrightarrow AB$ is normal.}\]
\end{thm}

\begin{proof}
First, recall that self-adjoint operators are maximally symmetric,
that is, if $T$ is self-adjoint and $S$ is symmetric, then $T\subset
S\Rightarrow T=S$ (we shall call this the MS-property). Second, if
$T$ is closed, then $T^*T$ and $TT^*$ are both self-adjoint (see
\cite{RUD} for both results).

Now, assume that $BA$ is normal. To show that $AB$ is normal observe
that it is first closed (which is essential) thanks to the
invertibility of $A$ and the closedness of $B$. We then have
\[(AB)^*AB=B^*A^*AB=B^*B\]
and
\[AB(AB)^*=ABB^*A^*\]
and we must show the equality of the two quantities. We have
\[A^*B^*BA\subset(BA)^*BA \text{ and } BB^*\subset BA(BA)^*.\]

By the closedness of $B$ and that of $BA$, and the MS-property
\[BB^*=BA(BA)^*.\]
 By the
normality of $BA$, we obtain
\[A^*B^*BA\subset BB^* \text{ or } A^*B^*B\subset BB^*A^*.\]
Hence
\[AB(AB)^*=ABB^*A^*\supset AA^*B^*B=B^*B=(AB)^*AB.\]
Therefore, and by the MS-property again, we must have
\[AB(AB)^*=(AB)^*AB.\]

To prove the converse, we may argue similarly with some minor
changes. Suppose that $AB$ is normal and let us show that $BA$ is
normal too.

Firstly, note that $BA$ is closed, but this time, since $B$ is
closed and $A$ is bounded.

Secondly, and by the normality of $AB$, we have
\[(AB)^*AB=B^*B=AB(AB)^*=ABB^*A^*\]
and hence
\[ABB^*=B^*BA.\]

Accordingly, we have
\[(BA)^*BA\supset A^*B^*BA=A^*ABB^*=BB^*\]
and
\[BA(BA)^*\supset BB^*.\]
By the MS-property, we must have
\[(BA)^*BA=BB^* \text{ and } BA(BA)^*=BB^*,\]
proving the normality of $BA$. The proof is complete.

\end{proof}

\bibliographystyle{amsplain}

\end{document}